\documentclass[12pt,A4paper,reqno]{amsart}
\usepackage{amssymb}
\usepackage{graphicx}
\usepackage{amscd}
\usepackage{amsmath}
\usepackage{amsfonts}
\usepackage{mathrsfs}

\providecommand{\U}[1]{\protect\rule{.1in}{.1in}}
\newcommand{\R}{\mathbb{R}}
\newcommand{\N}{\mathbb{N}}
\newcommand{\Z}{\mathbb{Z}}
\newcommand{\C}{\mathbb{C}}

\newcommand{\F}{\mathscr{F}}
\newcommand{\dq}{\partial_q}
\providecommand{\U}[1]{\protect\rule{.1in}{.1in}}
\providecommand{\U}[1]{\protect\rule{.1in}{.1in}}

\newcommand{\ds}{\displaystyle}
\newtheorem{theorem}{Theorem}[section]

\newtheorem{lemme}{Lemma} [section]
\newtheorem{coro}{Corollary} [section]
\newtheorem{remark}{Remark} [section]

\numberwithin{equation}{section}
\everymath{\displaystyle}

\begin{document}
\title[Real Paley-Wiener Theorem ]{Real Paley-Wiener Theorem for the Generalized Weinstein transform in quantum calculus}
\author[Y. Bettaibi \& H. Ben Mohamed]{}
\maketitle

\centerline{\bf Youssef Bettaibi$^1$}

\centerline{E-mail : youssef.bettaibi@yahoo.com}

\centerline{\bf Hassen Ben Mohamed$^1$}
\centerline{E-mail : hassenbenmohamed@yahoo.fr}

$^1${University of Gabes, Faculty of Sciences of Gabes, LR17ES11 Mathematics and Applications, 6072, Gabes, Tunisia.
.
}

\begin{abstract}
We first characterize the image of the compactly supported smooth even functions
under the q-Weinstein transform as a subspace of the Schwartz space. We then describe the space of
smooth $L_{\alpha, q, a}^{2}$-functions whose q-Weinstein transform has compact support as a subspace of the space of
$L_{\alpha, q, a}^{2}$-functions.
\end{abstract}
{\bf 2010 AMS Classification : } {33D15; 33E20; 33D60; 42B10}

\section{Introduction}
The original Paley-Wiener theorem \cite{RPW1} describes the Fourier transform of $L^{2}$ -functions
on the real line with support in a symmetric interval as entire functions of exponential type whose restriction to the real line are $L^{2}$-functions, which has proved to be a basic tool for transform in various set-ups. Recently, there has been a great interest
in the real Paley-Wiener theorem due to Bang \cite{RPW2}, in which the adjective "real" expresses that information about the support of the Fourier transform comes from growth
rates associated to the function $f$ on $\mathbb{R},$ rather than on $\mathbb{C}$ as in the classical "complex Paley-Wiener theorem". Bang \cite{RPW2} discovered a characterization of band-limited signals by
using a derivative operator, whose result can be rephrased as $f \in L^{2}(\mathbb{R})$ is band-limited
with bandwidth $\sigma$ if and only if $f$ infinitely differentiable, $\frac{d^{m}f}{d t^{m}} \in L^{2}(\mathbb{R})$ for all positive integers $m$ and $$\lim _{m \rightarrow \infty}\left\|\frac{\mathrm{d}^{m}}{\mathrm{d} t^{m}} f\right\|_{L^{2}(\mathbb{R})}^{1 / m}=\sigma=\sup \{|\xi|: \xi \in \operatorname{supp} \mathcal{F}( f)\},$$ where $\mathcal{F}( f)$ is the Fourier transform of $f.$ A wide number of papers have been devoted to the extension of the theory on higher dimensions and many other integral transforms (see\cite{RPW3}, \cite{RPW13}, \cite{RPW22}, and the references therein).\\

A class of Paley-Wiener theorems sitting inside the Schwartz space was obtained by Andersen
in \cite{RPW5}, where it is shown that the Fourier transform is a bijection between smooth functions
supported in $[-R, R]$ and the space of all Schwartz functions satisfying, for all $m \in \N$
$$
\sup _{x \in \R, n \in \mathbb{N}_{0}} R^{-n} n^{-N}(1+|x|)^{N}\left|\frac{d^{n}}{d x^{n}} f\right|<\infty.
$$
Following the classical theory, an element of $P W_{a}$ will be called
bandlimited signal. \\
papers have been devoted to the extension of the theory on many other transforms and different
classes of functions, for example, Hankel transform (see \cite{RPWH}) and the Weinstein transform (see \cite{RPWW})
.\\

In the literature, these theorems are known to hold for more general transforms in classical analysis as well as in Quantum Calculus for example
the Real Paley-Wiener Theorem for the q-Dunkl transform(see \cite{RPWD}), q-Hankel transform(see \cite{RPWH})\\

In \cite{youss}, we introduce a q-analogue of the Weinstein operator and we investigate its eigenfunction. Next, we study its associated Fourier transform which is a q-analogue of the Weinstein transform.\\
In this paper, we shall continue their work by giving two real Paley-Wiener theorems for the $q$-Weinstein transform. The first uses techniques due to Tuan and Zayed \cite{RPW}, in order to describe the image under the $q$-Weinstein transform $\F_{W}^{\alpha, q}$ of $L_{\alpha, q,a}^{2}$ (the space of square integrable functions on $B_{(0,a)}$ with respect to the measure $\left.x_{2}^{2 \alpha+1} d_{q} x_{1}d_{q} x_{2}, \alpha \geq-1 / 2\right) .$ The second characterizes the image of the compactly supported q-smooth functions domain under $\F_{W}^{\alpha, q}$

 This paper is organized as follows: in Section $2,$ we present some standard conventional notations used in the sequel. In Section $3,$ we will mention some results and definitions from the theory of $q$-Weinstein operator and $q$-Weinstein transform. All of these results can be found in \cite{youss}. Section 4 is devoted to study the real Paley-Wiener theorem for $q-L^{2}$-functions. Finally in Section $5,$ we give a real Paley-Wiener theorem for the $q$-Schwartz functions.
\section{Notations and preliminaries}
For the convenience of the reader, we provide in this section a
summary of the mathematical notations and definitions used in this
paper. We refer the reader to the general references \cite{GR} and
\cite{KC}, for the definitions, notations and properties of the
$q$-shifted factorials and the $q$-hypergeometric functions.
Throughout this paper, we assume $q\in]0, 1[$ and we  denote \\
$\bullet\ds~\R_q=\{\pm q^n~~;~~n\in\Z\}$,  $\ds \R_{q,+}=\{q^n~~;~~n\in\Z\}$.\\
$\bullet \ {\R_q^{2}}=\R_q\times\R_{q}$ and $ {\R_{q,+}^{2}}=\R_q\times\R_{q,+}$  \\
$\bullet ~x= (x_1,x_2)\in\R^2, -x =(-x_1,x_2)$ and $ \| x\|=\sqrt{x_1^{2}+x_2^{2}}$ \\
$\bullet$ For $a>0 $, $B_{(0,a)}=\Big\{x\in{\R_q^{2}},\quad \|x\|\leq a\Big\}$, $\quad B_{+(0,a)}=B_{(0,a)}\bigcap{\R_{q,+}^{2}} $\\

\subsection{Basic symbols}
For $x \in \C $, the $q$-shifted factorials are defined by
\begin {equation*}
 (x;q)_0=1;~~ (x;q)_n = \ds\prod _{k=0}^{n-1}(1-xq^k),~~
 n=1,2,...;~~  (x;q)_\infty = \ds\prod _{k=0}^\infty(1-xq^k).
\end {equation*}
We also denote
 \begin {equation*}
 [x]_q={{1-q^x}\over{1-q}},\quad ~ x\in \C\quad {\rm and}\quad [n]_q! ={{(q;q)_n}\over
 {(1-q)^n}},  \quad ~~ n\in \N.
\end {equation*}

\subsection{ Operators and elementary special  functions}~~\\
 The  $q$-Gamma function is given by (see  \cite{Jac} )
$$
\Gamma_q (x) ={(q;q)_{\infty}\over{(q^x;q)_{\infty}}}(1-q)^{1-x} ,
~~ ~~x\neq 0, -1 , -2 ,...
$$
It satisfies the  following relations \begin {equation*}
\label{gam1} \Gamma_q (x+1) =[x]_q \Gamma_q (x) ,~~ \Gamma_q (1) =
1~~ \hbox{ and  } \lim_{q\longrightarrow 1^-}\Gamma_q (x) = \Gamma(x) ,
\Re(x) >0.
\end {equation*}
The third Jackson's normalized Bessel function  is given by ( see \cite {Rubin})
\begin{equation}\label{j}
j_\alpha(x;q^2) =  \ds \sum_{n=0}^{+\infty}
(-1)^n \frac{\Gamma_{q^2}(\alpha+1)q^{n(n+1)}}{(1+q)\Gamma_{q^2}(\alpha+n+1)\Gamma_{q^2}(n+1)}x ^{2n},
\end{equation}
 the $q$-trigonometric functions $q$-cosine and $q$-sine are
defined by ( see \cite {Rubin})
\begin {equation}\label{cos}
 \cos(x;q^2)=j_{-\frac{1}{2}}\left(x ; q^{2}\right) \quad{\rm ,
}\quad\sin(x;q^2)=xj_{\frac{1}{2}}\left(x ; q^{2}\right)
\end {equation}
 and the $q$-analogue exponential function is given by ( see \cite {Rubin})
\begin{align}\label{exp}
e(z;q^2) =\cos(-iz;q^2)+i\sin(-iz;q^2)
\end{align}
These functions are absolutely convergent for all $z$ in the
plane and when $q$ tends to 1 they tend to the corresponding
classical ones  pointwise and uniformly on compacts.\\
Note that we have for all $x \in \R_q$ (see \cite{Rubin})
\begin{equation}\label{iexp}
 |j_\alpha(x;q^2)|\leq
\frac{1}{(q;q)_\infty} \quad \hbox{and} \quad   |\ e(ix;q^2)|\leq \ds \frac{2}{(q;q)_\infty}.
\end{equation}
The $q^2$-analogue differential operator is ( see \cite {Rubin})
\begin {equation*}\partial_q(f)(z)=\left\{\begin{array}{cc}
                   \ds\frac{f\left(q^{-1}z\right)+f\left(-q^{-1}z\right)-f\left(qz\right)+f\left(-qz\right)-2f(-z)}{2(1-q)z}
 & if~~z\neq 0 \\
                   \ds\lim_{x\rightarrow 0}\partial_q(f)(x)\qquad({\rm in}~~ \R_q) & if~~z= 0. \\
                 \end{array}\right.
\end {equation*}

\begin{remark} If $f$ is differentiable at $z$, then $\ds
\lim_{q\rightarrow1}\partial_q(f)(z)=f'(z)$.\\
A repeated application of the $q^2$-analogue differential operator
$n$ times is given by:$$ \partial_q^0f=f,\quad
\partial_q^{n+1}f=\partial_q(\partial_q^nf). $$
For $\beta=(\beta_1,\beta_2)\in \N\times\N$, we use the notation $$D_q^{\beta}=\partial_{x,q}^{\beta_1}\partial_{y,q}^{\beta_2} .$$
 The $q^2$-analogue Laplace operator or $q$-Laplacian is given by
$$\Delta_q= \partial_{x,q}^2+\partial_{y,q}^2.$$
\end{remark}
 The following lemma lists some useful computational
properties of $\dq$, and reflects the sensitivity of this operator
to parity of its argument. The proof is straightforward.
\begin{lemme}\label{ld}~~\\
1) We have \begin{align*}
  & \dq \sin(x;q^2)=\cos(x;q^2),\quad \dq \cos(x;q^2)=-\sin(x;q^2), \\
   & \dq e(x;q^2)=e(x;q^2) \quad\text{and}\quad
\partial_q j_\alpha(x;q^2)=-\frac{x}{[2\alpha +2]_q}j_{\alpha +1}(x;q^2).
\end{align*}
2)  For all function $f$ we have $$\ds \dq
f(z)=\frac{f_e(q^{-1}z)-f_e(z)}{(1-q)z}+\frac{f_o(z)-f_o(qz)}{(1-q)z}.$$
Here,  for a function $f$ defined on $\R_q$, $f_e$ and  $f_o$ are
its  even and odd parts respectively.\\
3) Let $f$ and $g$ two functions. \\
  \indent i) If $f$ even and $g$ odd we have
  $$
  \partial_q(fg)(z)=q\partial_q(f)(qz)g(z)+
  f(qz)\partial_q(g)(z) =\partial_q(g)(z))f(z)+
  qg(qz)\partial_q(f)(qz);$$
\indent ii) If $f$ and $g$ are even we have $$
  \partial_q(fg)(z)=\partial_q(f)(z)g(q^{-1}z)+
  f(z)\partial_q(g)(z).$$

iii) If $f$ and  $g $
  		 are odd:
\begin{equation*}
\dq (f g)(x)=\dq(f)(x) g\left(\frac{x}{q}\right)+f(x) \dq(g)(x)- \frac{f(x)}{x}\left(g\left(\frac{x}{q}\right)+g(x)\right)
\end{equation*}

\end{lemme}
By the use of the $q^2$-analogue differential operator $\dq$, we
note:\\
$ \bullet$ $ \mathscr{E}_{q}(\R_q^2)$, the space of functions $f$
defined on $\R_q\times\R_q$, satisfying for all $n\in\N$ and all $a \geq 0$,
$$
P_{n,a}(f) = \sup\left\{ |D _{q}^{\beta}f(x)|, \mid \beta\mid\leq n; x\in \mathbb{R}_{q}^2:\|x\|\leq a\right\}<\infty $$ and  $$ \lim_{x\rightarrow (0,0)}D _{q}^{\beta} f(x)\quad (\hbox{ in}\quad
\mathbb{R}_{q}^2)\qquad \hbox{
	exists}.$$
We provide it with the topology  defined by the semi
norms $P_{n,a}.$ \\
$ \bullet$ $\mathscr{E}_{\ast ,q}(\R_q^2)$, the space of functions in
$\mathscr{E}_{q}(\R_q\times\R_q)$, even with respect to the last
variable. \\
$ \bullet$ $ \mathscr{S}_{q}(\mathbb{R}_{q}^2)$, the space of functions $f$
defined  on $\mathbb{R}_{q}^2$ satisfying
$$\forall n\in\N,\quad P_{n,q}(f)=\sup_{x\in\mathbb{R}_{q}^2}\sup_{|\beta|\leq n}\left| D _{q}^\beta\left[\|x\|^{2n} f(x)\right]\right|<+\infty$$
and $$ \lim_{x\rightarrow (0,0)}D _{q}^{\beta} f(x)\quad ({\rm in}\quad
\mathbb{R}_{q}^2)\qquad {\rm
	exists}.$$
$\bullet$ $ \mathscr{S}_{\ast,q} (\mathbb{R}_{q}^2)$, the space of functions in
$\mathscr{S}_{q}(\R_q^2)$, even with respect to the last
variable. \\
$\bullet$ $ \mathscr{D}_{q}\left(\mathbb{R}_{q}^2\right),$ the space of the restrictions on $\mathbb{R}_{q}^2$ of infinity q-differentiable functions
on $\mathbb{R}_{q}^2$ with compact supports.\\
$\bullet$ $ \mathscr{D}_{a,q} (\mathbb{R}_{q}^2)$, the space function in
$ \mathscr{D}_{q}(\mathbb{R}_{q}^2),$ supported in $B_{(0,a)}$.\\
$\bullet$ $ \mathscr{D}_{\ast ,q} (\mathbb{R}_{q}^2)$, the space function in
$\mathscr{D}_{q}(\mathbb{R}_{q}^2) $, even with respect to the last variable. \\

$\bullet$ $ \mathscr{D}_{\ast ,a,q} (\mathbb{R}_{q}^2)$, the space function in
$ \mathscr{D}_{\ast ,q}(\mathbb{R}_{q}^2),$ supported in $B_{(0,a)}$.\\

 The $q$-Jackson integrals are defined
by (see \cite {Jac})  {\small\begin {equation}\label{int1}
\int_0^a{f(x)d_qx} =(1-q)a \sum_{n=0}^{\infty}q^nf(aq^n),\quad
\int_a^b{f(x)d_qx} =\int_0^b{f(x)d_qx}-\int_0^a{f(x)d_qx},
\end {equation}}
\begin {equation}\label{int2}
\int_{\R_{q,+}}f(x)d_qx=\int_0^{\infty}f(x)d_qx
=(1-q)\sum_{n=-\infty}^{\infty}q^nf(q^n)\quad
\end {equation}
and \begin {equation}\label{int2019}
\int_{\R_{q}}f(x)d_qx=\int_{-\infty}^{\infty}f(x)d_qx
=(1-q)\sum_{n=-\infty}^{\infty}q^nf(q^n) +(1-q)\sum_{n=-\infty}^{\infty}q^nf(-q^n),
\end {equation}
provided the sums converge absolutely. \\
The following simple result, giving  $q$-analogues of the
integration by parts theorem, can be verified by direct
calculation.
\begin{lemme}\label{ppar}~~\\
1) For $a>0$, if $\ds \int_{-a}^a (\dq f)(x)g(x)d_qx$ exists, then
{\small \begin{equation*} \int_{-a}^a (\dq
f)(x)g(x)d_qx=2\left[f_e(q^{-1}a)g_o(a)+f_o(a)g_e(q^{-1}a)\right]-\int_{-a}^a
f(x)(\dq g)(x)d_qx.
\end{equation*}}
2) If $\ds \int_{-\infty}^\infty (\dq f)(x)g(x)d_qx$ exists,
\begin{equation}\label{006}
\int_{-\infty}^\infty (\dq f)(x)g(x)d_qx=-\int_{-\infty}^\infty
f(x)(\dq g)(x)d_qx.
\end{equation}
\end{lemme}

Using the $q$-Jackson integrals, we note for $p>0 $ and  $\alpha \geq-\frac{1}{2}$,\\
 $\bullet$ $ \ds L_{\alpha ,q}^p(\R_{q,+}^{2})$ the Banach space constituted of   functions  such that
 $$
\|f\|_{L_{\alpha ,q}^p(\R_{q,+}^{2})}=\left(\int_{\R_{q,+}^{2}}|f(x_1,x_2)|^pd\mu_{\alpha,q}(x_1,x_2)\right)^{\frac{1}{p}}<\infty , \text {  if } p <\infty
,$$
where
 \begin{equation*}
  d\mu_{\alpha,q}(x_1,x_2)=x_{2}^{2\alpha+1}d_qx_1d_qx_2.
\end{equation*}
 $$
\|f\|_{L_q^\infty(\R_{q,+}^{2})}=\sup_{x\in\R_{q,+}^{2}}|f(x)|<\infty. $$

\begin{lemme}\label{l23}
	Let $f \in \mathscr{D}_{R,q}(\R_{q}^2) $. Then, for $n_1$, $n_2$, $p_1$, $p_2$ and $ p\in\N$  such that   $ p_1\leq p<n_1 $ and $p_2\leq p<n_2 $, then  the function $ (t_1,t_2)\mapsto D_{q}^{(p_1,p_2)}\Big(t_1^{n_1}t_2^{n_2} f\Big)\in\mathscr{D}_{\frac{R}{q^{p}},q}(\R_{q}^2)$  	
	\begin{equation}
\Big\Arrowvert D_q^{(p_1,p_2)}\Big(t_1^{n_1}t_2^{n_2} f\Big)\Big\Arrowvert_{L_q^\infty(\R_{q,+}^{2})}\leq C(R,p)\Big(\frac{R}{q^{2p}}\Big)^{n_1+n_2}\frac{(q^{n_1};q^{-1})_p(q^{n_2};q^{-1})_p}{(1-q)^{2p}}.
	\end{equation}
\end{lemme}
\begin{proof}
	Using Lemma \ref{ld}, we have
	$$ D_q^{(p_1,p_2)}\Big(t_1^{n_1}t_2^{n_2} f\Big)(t_1,t_2)=t_1^{n_1-p_1}t_2^{n_2-p_2}(-1)^{p_1+p_2}\frac{\left(q^{-n_1} ; q\right)_{p_1}\left(q^{-n_2} ; q\right)_{p_2}}{(1-q)^{p_1+p_2}} f_{p_1,p_2}(t_1,t_2),$$
	where $f_{p_1,p_2}$ is a function satisfying $\operatorname{supp} (f_{p_1,p_2}) \subset B(0,q^{-p}R)$ and
	$$
	\left\|f_{p_1,p_2}\right\|{_{L_q^\infty(\R_{q,+}^{2})}} \leq C_{p_1,p_2} \sum_{k_1=0}^{p_1}\sum_{k_2=0}^{p_2}\left\|D^{(k_1,k_2)} f\right\|_{L_q^\infty(\R_{q,+}^{2})}.
	$$
	So,
	\begin{align*}
	 &\Big\Arrowvert D_q^{(p_1,p_2)}\Big(t_1^{n_1}t_2^{n_2} f\Big)\Big\Arrowvert_{L_q^\infty(\R_{q,+}^{2})}\\&\leq C_{p_1,p_2} \sum_{k_1=0}^{p_1}\sum_{k_2=0}^{p_2}\left\|D_q^{(k_1,k_2)} f\right\|_{L_q^\infty(\R_{q,+}^{2})}(q^{-p}R)^{n_1+n_2-p_1-p_2}  \Big|\frac{\left(q^{-n_1} ; q\right)_{p_1}\left(q^{-n_2} ; q\right)_{p_2}}{(1-q)^{p_1+p_2}} \Big|\\
	 & \leq C_{p_1,p_2} \sum_{k_1=0}^{p_1}\sum_{k_2=0}^{p_2}\left\|D_q^{(k_1,k_2)} f\right\|_{L_q^\infty(\R_{q,+}^{2})}(q^{-p}R)^{n_1+n_2-p_1-p_2}  \Big|\frac{\left(q^{-n_1} ; q\right)_{p}\left(q^{-n_2} ; q\right)_{p}}{(1-q)^{2p}} \Big|\\
	 &= C_{p_1,p_2} \sum_{k_1=0}^{p_1}\sum_{k_2=0}^{p_2}\left\|D_q^{(k_1,k_2)} f\right\|_{L_q^\infty(\R_{q,+}^{2})}(q^{-2p}R)^{n_1+n_2}(q^{-p}R)^{-p_1-p_2}  q^{p(p-1)} \frac{\left(q^{n_1} ; q^{-1}\right)_{p}\left(q^{n_2} ; q^{-1}\right)_{p}}{(1-q)^{2p}} \\
	 & \leq C(R,p)\Big(\frac{R}{q^{2p}}\Big)^{n_1+n_2}\frac{(q^{n_1};q^{-1})
_p(q^{n_2};q^{-1})_p}{(1-q)^{2p}},
	  \end{align*}
	  where
	  $$C(R,p)=\max_{p_1,p_2\leq p} \Big\{C_{p_1,p_2} \sum_{k_1=0}^{p_1}\sum_{k_2=0}^{p_2}\left\|D_q^{(k_1,k_2)} f\right\|_{L_q^\infty(\R_{q,+}^{2})}(q^{-p}R)^{-p_1-p_2}  q^{p(p-1)}\Big\}.$$
	\end{proof}
\begin{coro} \label{RD}	
	Let $f \in \mathscr{D}_{R,q}(\R_{q}^2)  $. Then, for $p,n
	,i$ and $j$$  \in\N$   such that   $ i,j\leq p $,   there exists $C_{p,R}>0$ such that
		\begin{equation}
		\Big\Arrowvert D_q^{(2i,2j)}\Big(\|t\|^{2n} f\Big)\Big\Arrowvert_{L_q^\infty(\R_{q,+}^{2})}\leq C_{p,R}\Big(\frac{R}{q^{4p}}\Big)^{2n}\Big(\frac{(q^{2n};q^{-1})_{2p}}{(1-q)^{2p}}\Big)^{2}.
		\end{equation}
			
\end{coro}
\begin{proof}
		Thanks to Lemma \ref{l23}, we have
		\begin{equation}	\begin{split}
		\Big\Arrowvert D_q^{(2i,2j)}\Big(\|t\|^{2n} f\Big)\Big\Arrowvert_{L_q^\infty(\R_{q,+}^{2})}
		&\leq\sum_{n_1+n_2=2n}\binom{2n}{n_1}\Big\Arrowvert D_q^{(2i,2j)}\Big(t_1^{2n_1}t_2^{2n_2}  f\Big)\Big\Arrowvert_{L_q^\infty(\R_{q,+}^{2})}\\
		&\leq C(R,p)\Big(\frac{R}{q^{4p}}\Big)^{2n}\sum_{n_1+n_2=2n}\binom{2n}{n_1}
		\frac{(q^{2n_1};q^{-1})_{2p}(q^{2n_2};q^{-1})_{2p}}{(1-q)^{4p}}.
		\end{split}\end{equation}
		Hence,   the fact  that $((q^{n};q^{-1})_p)_{n>p}$ is  an increasing  sequence  achieves the proof.
	\end{proof}
\section{The $q$-Weinstein transform}
In \cite{youss}, a $q$-analogue of the Weinstein operator and its associated Fourier transform are introduced and studied.  In this section, we collect some of their basic properties.

\noindent$\bullet$ The $q$-Weinstein operator is given by
\begin{equation*}
  \triangle_{\alpha,q}=\partial_{q,x}^{2}+\frac{1}{ |y|^{2\alpha+1}}\partial_{q,y}( |y|^{2\alpha+1}\partial_{q,y})=\partial_{q,x}^{2}+\mathscr{B}_{{\alpha },q},\quad \alpha\geq-\frac{1}{2},
\end{equation*} where $\mathscr{B}_{{\alpha },q}$ is the $q$-Bessel operator defined in \cite{pwb}.\\

  \noindent$\bullet$ For all $f$, $g\in \mathscr{S}_{\ast,q} ({\R_q^{2}})$, we have
  \begin{equation}\label{labla}
   \int_{0}^{+\infty}\int_{-\infty}^{+\infty}\triangle_{{\alpha },q}f(x,y)g(x,y) y^{2\alpha+1}
 d_{q}x d_{q}y=\int_{0}^{+\infty}\int_{-\infty}^{+\infty}\triangle_{{\alpha },q}g(x,y)f(x,y) y^{2\alpha+1}
 d_{q}x d_{q}y.
  \end{equation}
  That is $\triangle_{{\alpha },q}$
  is self-adjoint.\\
 \noindent $\bullet$ For all  $\lambda=(\lambda_{1},\lambda_{1}) \in \mathbb{C}^2,$ the function
  \begin{equation}\label{ej}
\Lambda^{\alpha}_{q,\lambda}(x)=\Lambda^{\alpha}_{q}(\lambda_{1}x_1,\lambda_{2}x_2)=e(-i\lambda_{1}x_1;q^{2})j_{\alpha}(\lambda_{2}x_2;q^{2})
\end{equation}
  is the unique solution of the $q$-differential-difference equation:
  \begin{equation*}
\left\{
\begin{array}{lll}
\mathscr{B}_{{\alpha },q}u(x_1,x_2) & = & -\lambda_{2} ^{2}u(x_1,x_2), \\
&\\
\partial_{q,x_1}^{2}u(x_1,x_2)& = &-\lambda_{1} ^{2}u(x_1,x_2),\\
&\\
u(0,0) =1, & &\partial_{q,x_2}u\left( 0,0\right)   =  0,\qquad  \partial_{q,x_1}u(0,0)  =  -i\lambda_{1}.

\end{array}%
\right.
\end{equation*}
The function $\Lambda^{\alpha}_{q,\lambda}$ satisfies the
following properties  (see\cite{youss} and \cite{youss 1})
 \begin{enumerate}
  \item Let $\lambda=(\lambda_1,\lambda_2) \in \mathbb{C}^2$. For all $n ,p\in \mathbb{N}, $ we have
 \begin{equation}\label{06}
   \partial_{q}^n\left[\mathscr{B}_{{\alpha },q}^p(\Lambda^{\alpha}_{q,\lambda})\right]=\mathscr{B}_{{\alpha },q}^p\left[\partial_{q}^n(\Lambda^{\alpha}_{q,\lambda})\right]=(-i\lambda_1)^n(-i\lambda_2)^{2p}\Lambda^{\alpha}_{q,\lambda}.
 \end{equation}
 \item For $n \in \mathbb{N}$ and $ \lambda \in \mathbb{R}^2$,we have
 \begin{equation}\label{di}
  \forall x \in \R_{q}^{2} , \quad \triangle_{{\alpha },q}^n(\Lambda^{\alpha}_{q,\lambda})(x)=(-1)^n\parallel\lambda\parallel^{2n}\Lambda^{\alpha}_{q,\lambda}(x).
 \end{equation}
  \item For $\lambda=(\lambda_1,\lambda_2) \in \mathbb{C}^2$ and $\beta=(\beta_1,\beta_2)\in \mathbb{N}^2,$ we have
 \begin{equation}\label{6}
 \forall x \in \R_{q}^{2} , \quad\left|D_q^{\beta}\left(\Lambda^{\alpha}_{q,\lambda}\right)(x)\right|\leq\frac{4|\lambda_1|^{\beta_1}|\lambda_2|^{\beta_2}}{(q;q)_\infty^2} .
 \end{equation}
 In particular,
  \begin{equation}\label{7}\forall x \in \R_{q}^{2} , \quad \left|\Lambda^{\alpha}_{q,\lambda}(x)\right|\leq\frac{4}{(q;q)_\infty^2}.
 \end{equation}
  \item For all $\lambda \in \R_{q}^{2},$ we have $\Lambda^{\alpha}_{q,\lambda} \in  \mathscr{S}_{\ast,q} ({\R_q^{2}})$.\\
  \item For all $x$, $y\in \mathbb{R}_{q,+}^{2} $, we have
\begin{equation}\label{pro}
 \int_{\R_{q,+}^{2}}\Lambda^{\alpha}_{q,\lambda}(x){\Lambda^{\alpha}_{q,-\lambda}(y)}d_{q}\mu_{\alpha}(\lambda)
  =\left[2(1+q)^{\alpha-{\frac{1}{2}}}\Gamma_{q^{2}}(\frac{1}{2})\Gamma_{q^{2}}(\alpha+1) \right]^{2}\delta^{\alpha}_{x}(y),
\end{equation}
withe  $\delta^{\alpha}_{x},~\alpha\geq-\frac{1}{2}$, denotes the weighted Dirac-measure at $x \in \mathbb{R}_{q,+}^2 $ defined by
\begin{equation*}
 \forall  y \in\mathbb{R}_{q,+}^{2}, \quad \delta^{\alpha}_{x}(y)=\left\{
\begin{array}{ccc}
\left[ (1-q)^{2}|x_{1}|x_{2}^{2\alpha+2}\right]^{-1}, &if& x=y, \\
0, & &ifnot.
\end{array}
\right.
\end{equation*}
  \item The function $(\lambda,z)\mapsto \Lambda^{\alpha}_{q,\lambda}(z)$ has a unique extension to $\mathbb{C}^2 \times \mathbb{C}^2$ and we have
 \begin{equation*}
 \forall z,\lambda\in \mathbb{C}^2,  \quad \Lambda^{\alpha}_{q,\lambda}(z)= \sum_{n,m=0}^{\infty} v_{n,m}(-i\lambda_1z_1)^{n}(i\lambda_2z_2)^{2m},
 \end{equation*}
 where $$\forall n,m \in \N\quad v_{n,m}=a_n.b_m . $$\\
  \item For all $x \in \mathbb{R}_{q,+}^{2} \cap[-a, a]^{2}$, we have
$$
\forall z \in \mathbb{C}^2, \quad \left|\Lambda_{q, z}^{\alpha}(x)\right| \leq 4 e^{{2}a({1+\sqrt{q}})\|z\|}.
$$ \end{enumerate}
 \begin{lemme}\label{ldd}
 		Let $f \in \mathscr{D}_{*,R,q}(\R_{q}^2)  $ for $k \in\N$, we have
 		\begin{equation}
 			\Big\Arrowvert \triangle_{{\alpha },q}^k f\Big\Arrowvert_{L_q^\infty(\R_{q,+}^{2})}\leq C_{k}\max_{p_1,p_2\leq k}	\Big\Arrowvert D_q^{2(p_1,p_2)} f\Big\Arrowvert_{L_q^\infty(\R_{q,+}^{2})}
 		\end{equation}
 \end{lemme}
\begin{proof}
	Let $f \in \mathscr{D}_{*,R,q}(\R_{q}^2)$, from the Lemma \ref{ld}, we have
\begin{align*}
(\triangle_{\alpha,q}f)(x,y)&=\partial_{x,q}^2f(x,y)+q^{2\alpha+1}\partial_{y,q}^2f(x,y)-\frac{q[-2\alpha-1]_{q}}{y}\partial_{y,q}f(x,y)\\
&=\partial_{x,q}^2f(x,y)+q^{2\alpha+1}\partial_{y,q}^2f(x,y)-{q[-2\alpha-1]_{q}}\int_{0}^{1}\partial_{y,q}^{2}f(x,yt)d_qt.
\end{align*}
So, for $ k \in\N$, $k\geq 1$,  $\triangle_{\alpha,q}^{k}f \in \mathscr{D}_{*,R,q}(\R_{q}^2)$ and we have
 \begin{align*}
(\triangle_{\alpha,q}^k f)(x,y)&=\sum_{m=0}^{k}\left(\begin{array}{l}{k} \\ {m}\end{array}\right) \Big(D_{q}^{2(k-m,k)} f\Big)(x,y) +\sum_{j=1}^{2k-1} \int_{0}^{1} \cdots \int_{0}^{1} P_{2k-1}\left(t_{1}, \ldots, t_{j}\right)\\
&\times \Big(D_{q}^{2(k-m,k)} f\Big)\left(x,\mathrm{y.t}_{1} \cdots \mathrm{t}_{\mathrm{j}}\right) \mathrm{dt}_{1} \cdots \mathrm{dt}_{\mathrm{j}}+\int_{0}^{1} \cdots \int_{0}^{1} \mathrm{Q}_{2\mathrm{k}-1}\left(\mathrm{t}_{1}, \ldots, \mathrm{t}_{\mathrm{k}-1}\right)\\
&\times \Big(D_{q}^{2(k-m,k)} f\Big)\left({x,y.t}_{1} \cdots t_{\mathrm{k}}\right) t_{1} \cdots t_{\mathrm{k}},
 \end{align*}
 where $\mathrm{P}_{\mathrm{2k}-1}\left(t_{1}, \ldots, t_{\mathrm{j}}\right), \mathrm{j}=1,2, \ldots, \mathrm{2k}-1,$ and $\mathrm{Q}_{\mathrm{2k}-1}\left(t_{1}, \ldots, t_{\mathrm{2k}-1}\right)$ are polynomials of
 degree at most $\mathrm{2k}-1$ with respect to each variable.\\
Thus there exists a positive constant $C_k$ such that
\begin{equation}
 \Big|\triangle_{{\alpha },q}^k f(x,y)\Big|\leq C_{k}\max_{p_1,p_2\leq k}	\Big\Arrowvert D_q^{2(p_1,p_2)}\Big( f\Big)\Big\Arrowvert_{L_q^\infty(\R_{q,+}^{2})}
\end{equation}
\end{proof}
 The $q$-Weinstein transform $\mathscr{F}^{\alpha ,q}_W$ is defined on $L_{\alpha, q}^{1}\left(\mathbb{R}_{q,+}^{2}\right)$ by :
\begin{equation}\label{FB}
\mathscr{F}^{\alpha ,q}_W(f)(\lambda) = K_{\alpha ,q}
\ds\int_{\R_{q,+}^{2}} f(x,y)\Lambda^{\alpha}_{q,\lambda}(x,y)y^{2\alpha +1}d_qxd_qy
\end{equation} $where$  \begin{equation}\label{c} K_{\alpha ,q} =\frac{(1+q)^{\frac{1}{2}-\alpha}}{2\Gamma_{q^2}\left( \frac{1}{2}\right)\Gamma_{q^2}\left( \alpha+1\right)}.
\end{equation}
The $q$-Weinstein transformation satisfies the following properties
 \begin{enumerate}
   \item   For all $f \in L_{\alpha ,q}^1 (\R_{q,+}^{2})$, $ \F_W^{\alpha ,q}(f)\in L_{q}^\infty(\R_{q,+}^{2})$, and we have
  \begin{equation}\label{3.12}
\|\F_W^{\alpha ,q}(f)\|_{L_{q}^\infty(\R_{q,+}^{2})} \leq \frac{4K_{\alpha
,q}}{(q;q)^2_\infty} \| f \|_{L_{\alpha ,q}^1 (\R_{q,+}^{2})}\end{equation} and
\begin{equation*}
\lim_{\|\lambda\|\rightarrow \infty}\F_W^{\alpha ,q}(f)(\lambda)=0.
\end{equation*}
 \item Let $f \in \mathscr{S}_{\ast,q} ({\R_q^{2}})$. According to relations (\ref{06}), \eqref{di} and integration by parts, we have
 \begin{equation}\label{10}
 \F_W^{\alpha ,q}( \partial_{q}^n \mathscr{B}_{{\alpha },q}^p(f))(\lambda )=(i\lambda_1)^{n}(i\lambda_2)^{2p}
\F_W^{\alpha ,q}(f)(\lambda ),
\end{equation}

 \begin{equation}\label{9}
 \F_W^{\alpha ,q}(x_1^{n}x_2^{2p}f)(\lambda )= i^{n+2p}\partial_{q}^n \mathscr{B}_{{\alpha },q}^p\left(
\F_W^{\alpha ,q}(f)\right)(\lambda ),
\end{equation}

 \begin{equation}\label{FDlab}
 \F_W^{\alpha ,q}(\triangle_{ \alpha ,q}f)(\lambda )= -\parallel\lambda\parallel^2
\F_W^{\alpha ,q}(f)(\lambda ),\end{equation}

 \begin{equation*}
 \F_W^{\alpha ,q}(\parallel .\parallel^{2} f)(\lambda )= -\triangle_{{\alpha },q}\left(
\F_W^{\alpha ,q}(f)\right)(\lambda ).\end{equation*}
 \item For $f,g \in
L_{\alpha ,q}^1 (\R_{q,+}^{2}),$ we have
 \begin{equation}\label{symD}
 \int_{\R_{q,+}^{2}}\F_W^{\alpha ,q}(f)(\lambda)g(\lambda
)d\mu_{\alpha,q}(\lambda) =
\ds\int_{\R_{q,+}^{2}}f(\lambda)\F_W^{\alpha ,q}(g)(\lambda)\mu_{\alpha,q}(\lambda).
\end{equation}
 \end{enumerate}
 (See\cite{youss}).\\
\begin{theorem}~\\
 i) \underline{Plancherel formula} \\
For $\ds \alpha\geq -1/2$, the $q$-Weinstein transform $\F_W^{\alpha,
q}$ is an isomorphism from $\mathscr{S}_{\ast,q} ({\R_q^{2}})$ onto itself.
Moreover, for all $f\in\mathscr{S}_{\ast,q} ({\R_q^{2}})$, we have
\begin{equation}\label{pldun}
\|\F_W^{\alpha,q}(f)\|_{L_{\alpha ,q}^2 (\R_{q,+}^{2})}=\|f\|_{L_{\alpha ,q}^2 (\R_{q,+}^{2})}.
\end{equation}
ii) \underline{Plancheral theorem} \\
The $q$-Weinstein transform can be uniquely extended  to an isometric
isomorphism on $L_{\alpha,q}^2(\R_{q,+}^{2})$. ~~Its  inverse transform
${(\F_W^{\alpha ,q})}^{-1}$ is given by:
\begin{equation}\label{11}
    {(\F_W^{\alpha ,q})}^{-1}(f)(x) =  {K_{\alpha ,q}}\ds\int_{\R_{q,+}^{2}}
    f(\lambda) \Lambda _{q,\lambda} ^{\alpha }(-x)d\mu_{\alpha,q}(\lambda).
\end{equation}
\end{theorem}
\section{REAL PALEY-WIENER THEOREM FOR $ \ds L_{\alpha ,q}^2(\R_{q,+}^{2})$-FUNCTIONS}
For $a \in \mathbb{R}_{q,+},$ we introduce the Paley-Wiener space $P W_{q, \alpha, a}$ as
\begin{equation}
P W_{q, \alpha, a}=\left\{f \in \mathscr{E}_{\ast ,q}(\R_q^{2}): \forall n \in \mathbb{N},\triangle_{ \alpha ,q}^{n} f \in L_{\alpha ,q}^p(\R_{q,+}^{2})\text { and } \lim _{n \rightarrow+\infty}\left\|\triangle_{ \alpha ,q}^{n} f\right\|^{\frac{1}{2n}}_{L_{\alpha ,q}^2(\R_{q,+}^{2})} \leq a\right\}
\end{equation}

\text { The main result in this section will need the following lemma. }
\begin{lemme}\label{2}
Let $F$ be a function defined on $\mathbb{R}_{q,+}^{2},$ such that $ x\longmapsto\lVert x\lVert^{2n} F(x) \in  L_{\alpha ,q}^2(\R_{q,+}^{2})$ for all $n \in \mathbb{N}.$ Then
\begin{equation}
\lim _{n \rightarrow+\infty} \Big\lVert\lVert x\lVert^{2n} F\Big\lVert_{L_{\alpha ,q}^2(\R_{q,+}^{2})}^{\frac{1}{2n}}=\sup \Big\{\|x\|^2,\quad x \in \text {supp} \operatorname{(F)}\cap \mathbb{R}_{q,+}^{2}\Big\}
\end{equation}
\end{lemme}
\begin{proof}
	 The case $F=0$ is trivial, since in this case $supp (F)=\emptyset .$\\ Suppose now that
	$F \neq 0$ and define a measure $m_{\alpha,q}$ on $\mathbb{R}_{q,+}^2$ by
	$$
	d m_{\alpha,q}=\|F\|_{L_{\alpha ,q}^2(\R_{q,+}^{2})}^{-2}|F(x)|^{2} d\mu_{\alpha,q} (x)
	$$
	We have   $m_{\alpha,q}\left(\mathbb{R}^{2}_{q,+}\right)=1$ and
	$$
\Big\lVert	\left\|x\right\|^{2n} F\Big\lVert_{L_{\alpha ,q}^2(\R_{q,+}^{2})}^{\frac{1}{2n}}=\|F\|_{L_{\alpha ,q}^2(\R_{q,+}^{2})}^{\frac{1}{2n}}\Big\lVert\|x\|^2\Big\lVert_{L^{2 n}\left(\mathbb{R}_{q,+}^2, d m_{\alpha,q}\right)}^2.
	$$
	Moreover,
	$$
	\lim _{n \rightarrow+\infty}\Big\lVert\|x\|^2\Big\lVert^2_{L^{2 n}\left(\mathbb{R}_{q,+}^2, d m_{\alpha,q}\right)}=\Big\lVert\|x\|^2\Big\lVert_{L^{\infty}\left(\mathbb{R}_{q,+}^2, d m_{\alpha,q}\right)}
	$$
	and

\begin{align*}
\Big\lVert\|x\|^2\Big\lVert_{L^{\infty}\left(\mathbb{R}_{q,+}^2, d m_{\alpha,q}\right)}&=\sup \Big\{\|x\|^2,\quad  x \in  { supp(m_{\alpha,q}) }\Big\}\\&=\sup \Big\{\|x\|^2,\quad x \in  {supp} \operatorname{(F)}\cap \mathbb{R}_{q,+}^{2}\Big\}.
\end{align*}

	Finally, the fact
	$$
	\lim _{n \rightarrow+\infty}\|F\|_{L_{\alpha ,q}^2(\R_{q,+}^{2})}^{\frac{1}{2n}}=1
	$$
	gives the result.

	\end{proof}
\textbf{Notation :} For  $a>0,$ we denote by $L_{\alpha, q, a}^{2}$ the space of functions in $L_{\alpha, q}^{2}\left(\mathbb{R}_{q,+}^{2}\right)$ with
compact support in $B_{(0,a)} .$
\begin{theorem}  For any  $a\in \mathbb{R}_{q,+},$ the q-Weinstein transform $\F_{W}^{\alpha, q}$ is bijective from $L_{\alpha, q, a}^{2}$
	onto $P W_{q, \alpha, a} .$
\end{theorem}
\begin{proof}
	Let $f\in L_{\alpha, q, a}^{2}$. \\
From the properties of $\F_{W}^{\alpha, q},$ we get $\F_{W}^{\alpha, q}(f) \in L_{\alpha, q}^{2}(\mathbb{R}_{q,+}^2)$ and by definition of $f,$ we have for
all $n \in \mathbb{N}, t \mapsto \|t\|^{2n} f(t)$ belongs to $L_{\alpha, q}^{1}\left(\mathbb{R}_{q,+}^2\right) \cap L_{\alpha, q}^{2}\left(\mathbb{R}_{q,+}^2\right) .$ Then, a repeated application
of the operator $\triangle_{ \alpha ,q} $ to the $\F_{W}^{\alpha, q}(f)$ gives
\begin{equation}
\left(\triangle_{ \alpha ,q}^{n}\F_{W}^{\alpha, q}( f)\right)(x)=(-1)^{n} \F_{W}^{\alpha, q}\left(\|t\|^{2n} f\right)(x), n=0,1,... . \ldots
\end{equation}
So, the properties of $\F_{W}^{\alpha, q}$ imply that $\F_{W}^{\alpha, q}f \in \mathcal{E}_{*,q}\left(\mathbb{R}_{q}^2\right)$ and for all nonnegative integer $n,$
$\left(\Delta_{\alpha, q}^{n}\F_{W}^{\alpha, q}( f)\right) \in L_{\alpha, q}^{2}\left(\mathbb{R}_{q,+}^2\right) .$ Moreover, the Plancherel theorem gives
\begin{equation}
\left\|\Delta_{\alpha, q}^{n}\F_{W}^{\alpha, q} (f)\right\|_{L_{\alpha, q}^{2}\left(\mathbb{R}_{q,+}^2\right)}^{2}=\Big\lVert \|t\|^{2n} f\Big\lVert_{L_{\alpha, q}^{2}\left(\mathbb{R}_{q,+}^2\right)}^{2}=\int_{\R_{q,+}^2} \|t\|^{4 n}|f(t)|^{2}d\mu_{\alpha,q}(t)
\end{equation}
By using Lemma \ref{2}, we get
\begin{equation}
\lim _{n \rightarrow+\infty}\left\|\left(\Delta_{\alpha, q}^{n} \F_{W}^{\alpha, q}f\right)\right\|_{L_{\alpha, q}^{2}\left(\mathbb{R}_{q,+}^2\right)}^{\frac{1}{2n}}=\sup \Big\lbrace \|\lambda\|^2,\quad {\lambda \in \operatorname{supp}(f) \cap \mathbb{R}_{q,+}^2}\Big\rbrace  \leq a
\end{equation}
and $\F_{W}^{\alpha, q}(f) \in P W_{q, \alpha, a}$\\

$\text { Reciprocally, let } f \in P W_{q, \alpha, a} . \text { We have by the inversion formula }$
\begin{equation}
f(x)={K_{\alpha, q}} \int_{\R_{q,+}^2}( \F_{W}^{\alpha, q})^{-1}(f)(t) \Lambda _{q,x} ^{\alpha }(t)d\mu_{\alpha,q}(t).
\end{equation}
So for all $n\in\N$ we have
\begin{equation}
\left(\Delta_{\alpha, q}^{n} f\right)(x)=(-1)^n{K_{\alpha, q}} \int_{\R_{q,+}^2} \|t\|^{2n}( \F_{W}^{\alpha, q})^{-1}(f)(t) \Lambda _{q,x} ^{\alpha }(t)d\mu_{\alpha,q}(t)
\end{equation}
Using the Plancherel formula, we obtain
$$\Big\Arrowvert\left\|t\right\|^{2n} ( \F_{W}^{\alpha, q})^{-1}(f)\Big\Arrowvert _{L^2_{\alpha, q}(\R_{q,+}^2)}=\left\|\Delta_{\alpha, q}^{n}(f)\right\|_{L^2_{\alpha, q}(\R_{q,+}^2)}<\infty.$$
 Then for
 all $n \in \mathbb{N}, t \mapsto \|t\|^{2n} ( \F_{W}^{\alpha, q})^{-1}(f)(t)$ belongs to $L_{\alpha, q}^{2}\left(\mathbb{R}_{q,+}^2\right).$
  \\Finally, by Lemma $2,$ we deduce that $( \F_{W}^{\alpha, q})^{-1}(f)\in L_{\alpha, q, a}^{2}.$

\end{proof}
\begin{remark}
 We have for all $a \in \mathbb{R}_{q,+}$
$$
P W_{q, \alpha, a}=\left\{f: \forall x \in \mathbb{R}_{q}^{2}, f(x)=\frac{K_{\alpha, q}}{2} \int_{0}^{a}\int_{-a}^{a} g(t) \Lambda^{\alpha}_{q,x}d\mu_{\alpha,q}(t), \quad g \in L_{\alpha, q, a}^{2}\right\}
$$
Then, any element of $P W_{q, \alpha, a}$ is extendable to an entire function on $\mathbb{C}^2$ of exponential type.
\end{remark}

\section{REAL PALEY-WIENER THEOREM FOR FUNCTIONS IN THE $q$-SCHWARTZ SPACE}
For $m \in \mathbb{N},$ we define the real Paley-Wiener space $P W_{\alpha, {q}}^{m}$ by
\begin{equation}
P W_{\alpha, q}^{m}=\left\{f \in \mathscr S_{*,q}\left(\mathbb{R}_{q}^{2}\right): \exists a \in \mathbb{R}_{+}, \text { such that } \sup _{x\in \mathbb{R}_{q}^{2}, n \in \mathbb{N}\\ ,n \geq m} a^{-2n} B_{n, m, q}(1+\|x\|^{2})^{m}\left|\Delta_{\alpha, q}^{n} f(x)\right|<\infty\right\}
\end{equation}where $B_{n, m, q}=\Big(\frac{(1-q)^{2m}}{\left(q^{2n} ; q^{-1}\right)_{2m}}\Big)^{2}$
\begin{theorem}
 For $m \in \mathbb{N}, m> \alpha+\frac{3}{2},$ the $q$-Weinstein transform $\F_{W}^{\alpha, q}$ is a bijection
from $\mathscr{D}_{*,q}\left(\mathbb{R}_{q}^2\right)$ onto $P W_{\alpha, q}^{m}$
\end{theorem}

\begin{proof}
Let $f \in P W_{a, q}^{m} .$ There exist a positive real $a$ and a constant $C_{a, m}$ such that for all $x \in \mathbb{R}_{q}^2$ and all integer $n \geq m$
$$
\left|\Delta_{\alpha, q}^{n} f(x)\right| \leq C_{a, m} a^{2n} \frac{1}{B_{n, m, q}} \frac{1}{(1+\|x\|^2)^{m}}
$$
Consider $x \in \mathbb{R}_{q,+}^2$ outside of $B_{(0,a)}.$ We have
$$
\Big(\F_{W}^{\alpha, q}\Big)^{-1}\left(\Delta_{\alpha, q}^{n} f\right)(x)=(-1)^n \|x\|^{2n} \Big(\F_{W}^{\alpha, q}\Big)^{-1}(f)(x)=(-1)^n \|x\|^{2n} \F_{W}^{\alpha, q}(f)(-x)
$$
So
\begin{align*}
\Big|\F_{W}^{\alpha, q}(f)(-x)\Big|&=\Big|\Big(\frac{-1}{\|x\|^2}\Big)^{n} {K_{\alpha, q}} \int_{\R_{q,+}^2} \Delta_{\alpha, q}^{n} f(t) \Lambda^{\alpha}_{q,-x}(t)d\mu_{\alpha,q}(t)\Big|\\
& \leq \frac{}{}\frac{ 4C_{a, m}K_{\alpha, q}}{(q;q)_\infty^2B_{n, m, q}} \Big(\frac{a}{\|x\|}\Big)^{2n}  \int_{\R_{q,+}^2} \frac{t_2^{2\alpha+1}}{(1+\|t\|^2)^{m}}d_qt_1d_qt_2
\end{align*}
since $\|x\|>a,$ this last quantity clearly approaches zero as $n$ tends to $+\infty,$ it follows
that $\operatorname{supp}\left(\F_{W}^{\alpha, q}\right)^{-1}(f) \subset B_{(0,a)} .$
Finally, since $\F_{W}^{\alpha, q}$ is an isomorphism from $\mathscr S_{*,q}\left(\mathbb{R}_{q}^2\right)$ onto itself and $f \in \mathscr{S}_{*,q}\left(\mathbb{R}_{q}^2\right),$ we obitain
$\left(\F_{W}^{\alpha, q}\right)^{-1}(f) \in \mathscr{S}_{*,q}\left(\mathbb{R}_{q}^2\right),$ which implies that $\left(\F_{W}^{\alpha, q}\right)^{-1}(f) \in \mathscr{D}_{*,a,q}(\mathbb{R}_{q}^2) \subset \mathscr{D}_{*,q}\left(\mathbb{R}_{q}^2\right)$.

Conversely, let $f \in \mathscr{D}_{q}\left(\mathbb{R}_{q}^2\right) .$ There exists then $R \in \mathbb{R}_{q,+}$ such that
 $supp (f) \subset B_{(0,R)}.$ We have for $x \in \mathbb{R}_{q}^2$ and integer $n \geq m$
 $$
 \Delta_{\alpha, q}^{n}\left(\F_{W}^{\alpha, q}(f)\right)(x)=(-1)^{n} {K_{\alpha, q}} \int_{\R_{q,+}^2} \|t\|^{2n} f(t) \Lambda^{\alpha}_{q,x}(t)d\mu_{\alpha,q}(t)
 $$
 Then, from the properties of the $q$-Weinstein operator $\Delta_{\alpha, q},$ we have for $p \leq m$
 \begin{align*}
\Big| \|x\|^{2p}\Delta_{\alpha, q}^{n}\left(\F_{W}^{\alpha, q}(f)\right)(x)\Big|&=\Big|(-1)^{n+p} {K_{\alpha, q}} \int_{\R_{q,+}^2}\Delta_{\alpha, q}^{p}\Big( \|t\|^{2n} f(t)\Big) \Lambda^{\alpha}_{q,x}(t)d\mu_{\alpha,q}(t)\Big|\\
&\leq \frac{4K_{\alpha, q}}{(q;q)_\infty^2} \int_{\R_{q,+}^2}\Big|\Delta_{\alpha, q}^{p}\Big( \|t\|^{2n} f(t)\Big) \Big|d\mu_{\alpha,q}(t)
\end{align*}
 So, from the Lemma \ref{ldd} and the Corollary \ref{RD}, it exists  positive constant, $C_{p,R,\alpha}$, independent of $n$  such that
 \begin{align*}
\Big| \|x\|^{2p}\Delta_{\alpha, q}^{n}\left(\F_{W}^{\alpha, q}(f)\right)(x)\Big|&\leq C_{p,R,\alpha}\Big(\frac{R}{q^{4p}}\Big)^{2n}\Big(\frac{(q^{2n};q^{-1})_{2p}}{(1-q)^{2p}}\Big)^{2} \int_{B_{+(0,q^{-2p}R)}}t_2^{2\alpha+1}d_qt_2d_qt_1\\
&\leq (q^{-2p}R)^{2\alpha+3}C_{p,R,\alpha}\Big(\frac{R}{q^{4m}}\Big)^{2n}\Big(\frac{(q^{2n};q^{-1})_{2m}}{(1-q)^{2m}}\Big)^{2}.
\end{align*}
 Now, taking $$a=\frac{R}{q^{4m}}, \text{ and }  C_{a, m}= \sum_{p=0}^{m} \binom{m}{p} (q^{2p}R)^{2\alpha+3}C_{p,R,\alpha},$$ we obtain
 \begin{align*}
\Big| \Delta_{\alpha, q}^{n}\left(\F_{W}^{\alpha, q}(f)\right)(x)\Big|(1+\|x\|^2)^{m}
 \leq C_{a, m} a^{2n} \frac{1}{B_{n, m, q}}
\end{align*}
	\end{proof}
\textbf{{Example:}} In Section 3.2 of \cite{j} it is shown that
 for $\alpha>-1 / 2$ and $p \geqslant 1,$ the $q-j_{\alpha+p}$ Bessel function has the $q$ -integral representation of Sonine type
$$
j_{\alpha+p}\left(y ; q^{2}\right)=\int_{0}^{1}  W_{p-1}\left(t ; q^{2}\right) j_{\alpha}\left(y t ; q^{2}\right) t^{2 \alpha+1}d_{q} t
$$
 where
$$
W_{p-1}\left(x ; q^{2}\right)=\frac{\left(x^{2} q^{2} ; q^{2}\right)_{\infty}}{\left(x^{2} q^{2 p-1} ; q^{2}\right)_{\infty}}
$$

So,
\begin{align*}
  \Lambda^{\alpha+p}_{q}(x,y)&=e(-ix;q^{2})j_{\alpha+p}(y;q^{2})\\
   &=\int_{-1}^{1}\int_{0}^{1} W_{p-1}\left(t_2 ; q^{2}\right)\Lambda^{\alpha}_{q}(t_1x,yt_2) \delta_{1}(t_{1}) d\mu_{\alpha,q}(t_1,t_2)
\end{align*}
where
\begin{equation*}
\forall  t \in\mathbb{R}_{q}, \quad \delta_{1}(t)=\left\{
\begin{array}{ccc}
\frac{1}{1-q}, &if& t=1, \\
0, & &ifnot.
\end{array}
\right.
\end{equation*}
As a result, $(x,y)\longmapsto\Lambda^{\alpha+p}_{q}(x,y) \in P W_{q, 1}^{\alpha}$ and satisfies
 \begin{align*}
\Big| \Delta_{\alpha, q}^{n}\left(\Lambda^{\alpha+p}_{q}\right)(x,y)\Big|
\leq  \frac{C_{1, m}}{B_{n, m, q}}\frac{1}{(1+\|(x,y)\|^2)^{m}}.
\end{align*}

\end{document}